\documentclass{amsart}
\usepackage{amsfonts}

\newtheorem{theorem}{Theorem}
\theoremstyle{plain}

\newtheorem{proposition}{Proposition}

\numberwithin{equation}{section}

\begin{document}
\title[Short Title]{A note on the points with dense orbit under $\times 2$ and $\times 3$ maps
}
\author{Author One}
\address[A. One and A. Two]{Author OneTwo address line 1\\
Author OneTwo address line 2}
\email[A. One]{aone@aoneinst.edu}
\urladdr{http://www.authorone.oneuniv.edu}
\thanks{Thanks for Author One.}
\author{Author Two}
\curraddr[A. Two]{Author Two current address line 1\\
Author Two current address line 2}
\email[A.~Two]{atwo@atwoinst.edu}
\urladdr{http://www.authortwo.twouniv.edu}
\thanks{Thanks for Author Two.}
\author{Author Three}
\address[A. Three]{Author Three address line 1\\
Author Three address line 2}
\urladdr{http://www.authorthree.threeuniv.edu}
\date{December 26, 1997}
\subjclass[2000]{Primary 05C38, 15A15; Secondary 05A15, 15A18}
\keywords{Hausdorff dimension, dense of the orbit}
\thanks{This paper is in final form and no version of it will be submitted
for publication elsewhere.}

\begin{abstract}
It was conjectured by Furstenberg that for any $x\in [0,1]\backslash
Q$, $$ \dim_H \overline{\{2^nx ({\text{mod}}\ 1): n\ge 1\}}+ \dim_H
\overline{\{3^nx ({\text{mod}}\ 1): n\ge 1\}}\ge 1.
$$ When $x$ is a normal number, the above result holds trivially. In this note, we give explicit non-normal numbers for which the above dimensional formula holds.
\end{abstract}
\maketitle

\emph {\section{\textbf{Introduction}}}
\par Let $b\geq2$ be an integer and $T_b$ be the $b$-ary expansion given by $$
Tx= bx\ ({\text{mod}}\ 1).$$
 $x$ is a real number given by its $b-$ary expansion, that is,
 $$x=\lfloor x\rfloor+\sum_{k\geq1} \frac{a_k}{b^{k}}=\lfloor x\rfloor+0.a_1a_2\cdots$$

where the digit $a_1,a_2,\cdots$ are integers from $\{0,1,\ldots
b-1\}$ and an infinity of the $a_k$ are not equal to $b-1$. Set
$$A_b(d,N,x):=Card\{j: 1\leq j\leq N, a_j=d\}.$$
More generally, for a block $D_k=d_1\cdots d_k$ of $k$ digits from
$\{0,1,\cdots,b-1\}$, set
$$A_b(D_k,N,x):=Card\{j: 0\leq j\leq N-k, a_j+1=d_1,\cdots,a_{j+k}=d_k\}$$

A real number $x$ is normal to base $b$, if and only if, for every
$k\geq1$, every block of $k$ digits from $\{0,1,\ldots b-1\}$ occurs
in the $b-$ary
 expansion of $x$ with the same frequency $\frac{1}{b^{k}}$, that
 is, if and only if, $$\
 \lim_{n\rightarrow\infty}\frac{A_b(D_k,N,x)}{N}=\frac{1}{b^{k}}$$
 for every $k\geq$ and every block $D_k$ of $k$ digits from $\{0,1,\ldots
 b-1\}$.

 Given one integer $b\ge 2$, the dynamical properties of  the orbit
$T_b^n(x)$ were well studied in the literature and are clearly known
at least for almost all $x$. However, as far as two integers $b_1$
and $b_2$ are concerned (assume that they are multiplicatively
independent), the properties of the orbits $T_{b_1}^x(x)$ and
$T^n_{b_1}(x)$ with a common point $x$ is far from being well
understood. Very little is known on this topic. Intuitively, if the
expansion of $x$ under $T_{b_1}$ is simple, it cannot be too simple
under $T_{b_2}$ (for some evidence see \cite{Bu1}). And the
normality of a number is closely related to his expansion to
different bases(\cite{PH},\cite{BBS}), nowdays, the investigate of
normal number has been used  to resolve the problems in dynamical
system and ergodic.

Furstenberg conjectured that \cite{Fu}: for any $x\in
[0,1]\backslash Q$,
$$ \dim_H \overline{\{2^nx ({\text{mod}}\ 1): n\ge 1\}}+ \dim_H
\overline{\{3^nx ({\text{mod}}\ 1): n\ge 1\}}\ge 1.
$$
It is clear that this is true for almost all $x$, since almost all
$x$ are normal numbers. So, a natural question is to ask, besides
normal numbers, can one give explicit examples fulfilling the above
conjecture. In this short note, we proof there exist a set $E$ of
full dimension which make Furstenberg conjectured hold.

\section{\textbf{Some notation}}
\par
We use $C_1$ and $C_2$ to denote the collection of finite $2$-adic
and $3$-adic words respectively, i.e.
$$C_{1}=(\bigcup\limits_{n\geq 1}\{0,1\}^{n}), \ \
C_{2}=(\bigcup\limits_{n\geq 1}\{0,1,2\}^{n}).$$  Let
 $\mathcal{A}=(w_1, v_1, w_2, v_2, \cdots)$, where $w_i\in
C_1$ and $v_i\in C_2$, be an enumeration of the words in $C_1$ and
$C_2$ such that each word in $C_1$ and $C_2$ appears infinitely many
times in $\mathcal{A}$.

Let
$$T_2 x=2x~({\text{mod}}~ 1) ~~~\mbox{and}~~~ T_3 x=3x~({\text{mod}}~ 1),  x\in
[0,1]$$  be the $\times 2$ and $\times 3$ maps.

\begin{itemize}
\item $|w|$: the length of the word $w$;

\item $\lfloor\xi\rfloor$: the integer part of a real number $\xi$;


\item $0^m$: a word of length $m$ composed by $0$.

\item $r_i$: the integer $\Big\lfloor |w_{i}|+(2+|v_{i}|)\log_{2}3 \Big\rfloor+2$;
\end{itemize}

For $u\in C_1$ and $v\in C_2$, the $2$-adic cylinder and $3$-adic
cylinder are defined as
\begin{align*}&[u]_{2}:=\{x\in [0,1]: {\text{the}}\
 2{\text{-adic expansion of}}\  x\  {\text{starts with}}\ u\},\\
& [v]_{3}:=\{x\in [0,1]: {\text{the}}\ 3{\text{-adic expansion of}}\
x\  {\text{starts with}} \ v\}.
\end{align*}

\section{\textbf{Main results}}
\begin{theorem} \label{3.1}
       There exist a Cantor set $E$ composed of nonnormal numbers such that for all $x\in E$,
       \begin{description}
         \item[1] $\overline{\{2^n x({\rm{mod}}\ 1): n\geq 1\}}=[0,1]$ and $\overline{\{3^n x({\rm{mod}}\ 1): n\geq
         1\}}=[0,1]$.
         \item[2] $\dim_{H}E=1$.
       \end{description}
       \end{theorem}
The construction of the Cantor set $E$ is divided into three steps.
Recall that $\mathcal{A}=\{w_1,v_1,w_2,v_2,\cdots\}$ be an
enumeration of the elements in $w\in C_1$ and $v\in C_2$ such that
each element in $C_1\cup C_2$ appears infinitely many times.

Fix an integer $m\ge 1$. Choose a sequence of integers
$\{\ell_k\}_{k\ge 1}, \{n_k\}_{k\ge 1}$ such that
\begin{equation}\label{f4} n_k=\ell_{k}\cdot m,\ \lim_{k\to
\infty}\frac{r_1+\cdots+r_k}{ n_{k}}=0.
\end{equation}
 \par For each $k\ge 1$, write
\begin{equation}\label{f2}\widetilde{C_1}^{n_k}:=\Big\{(i_1,\cdots,i_{n_k})\in
C_1: i_{\ell m+1}=i_{\ell m+m}=1, 0\le
\ell<\ell_k\Big\}.\end{equation}

 $\mathbf{Step\ 1}$: the first level of the Cantor set.

 For each
$(i_{1}^{(1)} \cdots i_{n_{1}}^{(1)})\in \widetilde{C_1}^{n_1}$,
consider the $2$-adic cylinder $$ [i_{1}^{(1)} \cdots
i_{n_{1}}^{(1)}, w_1]_2.
$$ Let $r_1$ be the integer such that $$
3\left(\frac{1}{3}\right)^{r_1}\le
\left(\frac{1}{2}\right)^{n_1+|w_1|}<
3\left(\frac{1}{3}\right)^{r_1-1}.
$$ Then the $2$-adic cylinder $[i_{1}^{(1)} \cdots i_{n_{1}}^{(1)}, w_1]_2$ will contain at least one $3$-adic cylinder of order $r_1$. Denote by $[\epsilon_1]_3$ such a $3$-adic cylinder (if there are  many, just choose one).

It should be mentioned that $r_1$ does not depend  on $(i_{1}^{(1)}
\cdots i_{n_{1}}^{(1)})$ but $\epsilon_1$ does. This dependence will
not play a role in the following argument, thus will not be
explicitly addressed.

Now consider the $3$-adic subcylinder $[\epsilon_1, v_1]_3$
contained in $[\epsilon_1]_3$. Similarly, let $t_1$ be the integer
such that
$$
2\left(\frac{1}{2}\right)^{t_1}\le
\left(\frac{1}{3}\right)^{r_1+|v_1|}<
2\left(\frac{1}{2}\right)^{t_1-1}.
$$ Then choose a $2$-adic cylinder $[\eta_1]_2$ of order $t_1$ contained in $[\epsilon_1, v_1]_3$, so contained in the $2$-adic cylinder $[i_{1}^{(1)} \cdots i_{n_{1}}^{(1)}, w_1]_2$. Thus we can write this $2$-adic cylinder $[\eta_1]_2$ as $$
[\eta_1]_2=[i_{1}^{(1)} \cdots i_{n_{1}}^{(1)}, w_1, \tilde{v_1}]_2,
$$ to emphasize its dependence on $v_1$.
Now we have the following inclusions: \begin{equation}\label{f1}
[i_{1}^{(1)} \cdots i_{n_{1}}^{(1)}, w_1, \tilde{v_1}]_2\subset
[\epsilon_1,v_1]_3\subset [i_{1}^{(1)} \cdots i_{n_{1}}^{(1)},
w_1]_2.
\end{equation}  A simple calculation can give us an estimation on the integer $t_1$:
\begin{equation*}
n_1+|w_1|+(|v_1|+1)\frac{\log 3}{\log 2}+1\le t_1<
n_1+|w_1|+(|v_1|+2)\frac{\log 3}{\log 2}+2.
\end{equation*}

The first level of $E$ is defined as $$ F_1=\bigcup_{(i_{1}^{(1)}
\cdots i_{n_{1}}^{(1)})\in \widetilde{C_1}^{n_1}} [i_{1}^{(1)}
\cdots i_{n_{1}}^{(1)}, w_1, \tilde{v_1}]_2,
$$ which is a collection of $2$-adic cylinders of order $t_1$.

 $\mathbf{Step\ k}$: the $k$-level of the Cantor set.

 This is an inductive step. Assume that the $(k-1)$th level $F_{k-1}$ has been constructed, which is a collection of $2$-adic cylinders $[\eta_{k-1}]$ of order $t_{k-1}$. Now we construct the $k$th level $F_k$.

 Fix an element $[\eta_{k-1}]_{2}\in F_{k-1}$. For each   $(i_1^{(k)}\cdots i_{n_k}^{(k)})\in \widetilde{C_1}^{n_k}$, consider the $2$-adic cylinder
$$[\eta_{k-1}, i_1^{(k)}\cdots i_{n_k}^{(k)}, w_k]_{2}.$$
Similar as in Step 1, let $r_k$ be the integer such that $$
3\left(\frac{1}{3}\right)^{r_k}\le
\left(\frac{1}{2}\right)^{t_{k-1}+|w_k|}<
3\left(\frac{1}{3}\right)^{r_k-1}.
$$ Then there is a $3$-adic cylinder, denoted by $[\epsilon_k]_3$, of order $r_k$ contained in the 2-adic cylinder $[\eta_{k-1}, i_1^{(k)}\cdots i_{n_k}^{(k)}, w_k]_{2}$ (if there are many, just choose one).

It should also be mentioned that $r_k$ does not depend on on
$(i_{1}^{(j)} \cdots i_{n_{j}}^{(j)})$ for $1\le j\le k$ but
$\epsilon_k$ does. This dependence will not play a role in the
following argument, thus will not be explicitly addressed.

Now consider the $3$-adic subcylinder $[\epsilon_k, v_k]_3$
contained in $[\epsilon_k]_3$. Let $t_k$ be the integer such that
$$
2\left(\frac{1}{2}\right)^{t_k}\le
\left(\frac{1}{3}\right)^{r_k+|v_k|}<
2\left(\frac{1}{2}\right)^{t_k-1}.
$$ Then choose a $2$-adic cylinder $[\eta_2]_2$ of order $t_k$ contained in $[\epsilon_{k}, v_k]_3$, so contained in the $2$-adic cylinder $[\eta_{k-1}, i_{1}^{(k)} \cdots i_{n_{k}}^{(k)}, w_k]_2$. Thus we can write this $2$-adic cylinder $[\eta_k]_2$ as $$
[\eta_k]_2=[\eta_{k-1}, i_{1}^{(k)} \cdots i_{n_{k}}^{(k)}, w_k,
\tilde{v_k}]_2,
$$ to emphasize its dependence on $v_k$.
Now we have the following inclusions: \begin{equation}\label{ff1}
[\eta_{k-1}, i_{1}^{(k)} \cdots i_{n_{k}}^{(k)}, w_k,
\tilde{v_k}]_2\subset [\epsilon_k,v_k]_3\subset [\eta_{k-1},
i_{1}^{(k)} \cdots i_{n_{k}}^{(k)}, w_k]_2.
\end{equation} A simple calculation can give us an estimation on the integer $t_k$ (replace the role of $n_1$ by $t_{k-1}$, $|w_1|$ by $|w_k|$, $|v_1|$ by $|v_k|$):
\begin{equation}\label{ff2}
t_{k-1}+|w_k|+n_{k}+(|v_k|+1)\frac{\log 3}{\log 2}+1\le t_k<
t_{k-1}+|w_k|+n_{k}+(|v_k|+2)\frac{\log 3}{\log 2}+2.
\end{equation}

Let $p_i=|w_k, \tilde{v}_k|=t_k-t_{k-1}-n_k\leq r_i$. Thus it
follows from (\ref{ff2}) and (\ref{f4}) that
\begin{equation}\label{f3}  \
\lim_{k\to\infty} \frac{|w_1,\tilde{v}_1|+\cdots
+|w_k,\tilde{v}_k|}{n_1+\cdots +n_k}\leq \
\lim_{k\to\infty}\frac{r_1+\cdots+r_k}{n_1\cdots +n_{k}}=0.
\end{equation}

The $k$th level of $E$ is defined as $$
F_k=\bigcup_{[\eta_{k-1}]_2\in F_{k-1}}\bigcup_{(i_{1}^{(k)} \cdots
i_{n_{k}}^{(k)})\in \widetilde{C_1}^{n_k}} [\eta_{k-1}, i_{1}^{(k)}
\cdots i_{n_{k}}^{(k)}, w_k, \tilde{v_k}]_2,
$$ which is a collection of $2$-adic cylinders of order $t_k$.

Finally, the desired Cantor set $E$ is defined as $$ E=\bigcup_{m\ge
1}E_m, \
  E_m=\bigcap_{k\ge 1}\bigcup_{[\eta_k]_2\in F_k} [\eta_k]_2.
$$

The following three propositions correspond to the three items in
Theorem \ref{3.1}.

\begin{proposition}
Every point in $E_m$ is nonnormal.
\end{proposition}
\begin{proof}
By the construction of $E_m$, for each $x\in E_m$, we express $x$ by
its 2-adic expansion as follows: $$ x=[i_{1}^{(1)} \cdots
i_{n_{1}}^{(1)}, w_1, \tilde{v}_1,\cdots, i_{1}^{(k)} \cdots
i_{n_{k}}^{(k)}, w_k, \tilde{v_k}, \cdots]_2,
$$ where $(i_{1}^{(k)} \cdots i_{n_{k}}^{(k)})\in \widetilde{C_1}^{n_k}$.

We count the frequency of the word $0^{2m}$ occurring in the 2-adic
expansion of $x$. Recall the definition of $\widetilde{C_1}^{n_k}$.
Among the first $t_{k-1}+n_k$ terms in the 2-adic expansion of $x$,
the number of the occurrence of $0^{m}$ will be less than
$p_1+\ldots+p_{k-1}$, Thus the frequency of the word $0^m$ occurring
in the 2-adic expansion of $x$ is less than $$ \liminf_{k\to
\infty}\frac{p_1+\ldots p_{k-1}}{n_k}=0.
$$ This shows that $x$ is nonnormal.

This shows that $x$ is nonnormal.
\end{proof}

\begin{proposition} For any $x\in E$, $$
\overline{\{2^nx ({\rm{mod}}\ 1): n\ge 1\}}=[0,1], \
\overline{\{3^nx ({\rm{mod}}\ 1): n\ge 1\}}=[0,1].
$$\end{proposition}
\begin{proof}
Fix an $x\in E$. It is sufficient to show that for any $2$-adic
finite word $w\in C_1$ and $v\in C_2$, they appear infinitely often
in the $2$-adic expansion and the $3$-adic expansion of $x$
respectively.

Recall the definition of $\mathcal{A}$. It suffices to show that for
each $k\ge 1$, $w_k$ and $v_k$ appear in the 2-adic expansion  and
the 3-adic expansion of $x$ respectively. This follows from
(\ref{ff1}) by noticing that $$ T_3^{r_k}x\in [v_k]_3, \
T_2^{t_{k-1}+n_k}(x)\in [w_k]_2.
$$
\end{proof}

\begin{proposition}
For each $m\ge 2$, $\dim_HE_m\ge 1-2/m$.
\end{proposition}
\begin{proof}
Recall the definition of $\widetilde{C}_1^{n_k}$ (\ref{f2}). Write
the points in $E_m$ by its 2-adic expansion: $$ \Big\{[i_{1}^{(1)}
\cdots i_{n_{1}}^{(1)}, w_1, \tilde{v}_1,\cdots, i_{1}^{(k)} \cdots
i_{n_{k}}^{(k)}, w_k, \tilde{v_k}, \cdots]_2,\ \ (i_{1}^{(k)} \cdots
i_{n_{k}}^{(k)})\in \widetilde{C_1}^{n_k}, k\ge 1 \Big\}.
$$

Now we will distribute a measure uniformly distributed on the
cylinders with nonempty intersection with $E_m$. Thus we count the
number of $t$-th cylinders which have nonempty intersection with
$E_m$ for each $t\ge 1$. We denote this number by $b_t$.

Recall the construction of the $k$th level $F_k$ of the Cantor set
$E_m$. We know that among the digit sequence of a point in $E_m$, at
the positions $$ \Big\{t_k+\ell m+1, t_k+(\ell+1)m,
t_k+n_{k+1}+1,\cdots, t_k+n_{k+1}+p_{k+1}, 0\le \ell<\ell_{k+1},
k\ge 0\Big\},
$$ the digits are fixed, while at other positions, the digits can be chosen arbitrarily in $\{0,1\}$. Thus the numbers $b_t$ can be given as follows.

(i). When $t\le t_1=n_1+p_1=\ell_1 m+p_1$, \begin{eqnarray*}
b_t=\left\{
      \begin{array}{ll}
        2^{t-2\ell}, & \hbox{when $t=\ell m$ for $0\le \ell <\ell_1$;}\\
2^{t-2\ell-1}, & \hbox{when $\ell m< t< (\ell+1)m$ for $0\le \ell<\ell_1$;} \\
        2^{\ell_1 m-2\ell_1}, & \hbox{when $\ell_1 m\le t\le n_1+p_1$.}
      \end{array}
    \right.
\end{eqnarray*}

(ii). When $t_{k}<t\le t_{k+1}=t_k+n_{k+1}+p_{k+1}=t_k+\ell_{k+1}
m+p_{k+1}$, \begin{eqnarray*} b_t=\left\{
      \begin{array}{ll}
\prod_{j=1}^{k}2^{\ell_j m-2\ell_j}\cdot 2^{t-t_k-2\ell}, & \hbox{when $t=t_k+ \ell m$ for $0\le \ell<\ell_{k+1}$;}\\
        \prod_{j=1}^{k}2^{\ell_j m-2\ell_j}\cdot 2^{t-t_k-2\ell-1}, & \hbox{when $\ell m< t< (\ell+1)m$ for $0\le \ell<\ell_{k+1}$;} \\
        \prod_{j=1}^{k+1}2^{\ell_j m-2\ell_j}, & \hbox{when $t_k+n_{k+1}\le t\le t_k+n_{k+1}+p_{k+1}$.}
      \end{array}
    \right.
\end{eqnarray*}

Thus for each 2-adic cylinder $I_t$ with nonempty intersection with
$E_m$, its measure can be given explicitly as $$ \mu(I_t)=b_t^{-1}.
$$

As a consequence, together with (\ref{f3}), we have
 \begin{equation}
\liminf_{t\to \infty}\frac{\log\mu(I_{t})}{\log|I_t|} =\liminf_{t\to
\infty}\frac{\log b_t}{t\log 2}\ge 1-\frac{2}{m}.
\end{equation}

\par So, using Billingsley's lemma \cite{BP}: Let $A\subset[0,1]$ be Borel and let $\mu$ be a finite Borel measure on $[0,1]$. Suppose $\mu(A)\geq0$.
If$$\alpha\leq\liminf_{n\to
\infty}\frac{log\mu(I_n(x))}{log|I_n(x)|}\leq\beta$$ for all $x\in
A$, then$ \alpha\leq dim(A)\leq \beta$. we conclude that
$$\dim_HE_m\geq1-\frac{2}{m}$$
and then $\dim_HE=1$.
\end{proof}


\begin{thebibliography}{99}


\bibitem{Fu}
     \newblock  Harry Furstenberg.
     \newblock Intersetions of cantor sets and transversality of
semigroups
     \newblock \emph{Sympos. Salomon. Bochner, princeton univ, princeton, N.
J. },  (1969), 41--59.







 \bibitem{PH}
     \newblock  Peter Hertling, 
     \newblock Simply normal numbers to different bases.,
     \newblock \emph{Journal of Univesal Computer Science}, \textbf{8(2)} (2002), 235-242(electronic).

\bibitem{Bu1}
\newblock Bugeaud,Yann,
\newblock  {Distribution modulo one and diophantin approximation}
\newblock \emph {Cambridge tracts in mathmatics }, {2012}



\end{thebibliography}
\end{document}